\documentclass[12pt,twoside,a4paper]{article}
\usepackage{amsmath}
\usepackage{amssymb}
\usepackage{graphicx}
\usepackage{syntonly}
\usepackage{amsthm}
\usepackage{newclude}
\usepackage{stmaryrd}
\usepackage[margin=1.2in]{geometry}
\usepackage{imakeidx}
\usepackage[nottoc,numbib]{tocbibind}

\usepackage{fancyhdr}
\usepackage[draft]{optional}
\usepackage[british,american]{babel}
\usepackage{lipsum}
\usepackage{mathrsfs}
\usepackage{amsfonts, wasysym, caption}
\usepackage{paralist}
\usepackage{comment}
\usepackage{amsopn}
\usepackage{accents,bbm,bibgerm,eucal,paralist,url,verbatim,wasysym}
\usepackage[normalem]{ulem}
\usepackage{
	latexsym,
	nicefrac
}
\usepackage[usenames]{color}
\usepackage{tikz}
\usetikzlibrary{decorations.pathreplacing}
\usepackage{url}
\usepackage[multiple]{footmisc}

\definecolor{darkgreen}{rgb}{0,0.5,0}
\definecolor{darkred}{rgb}{0.7,0,0}
\usepackage[colorlinks, 
citecolor=darkgreen, linkcolor=darkred
]{hyperref}

\DeclareMathAlphabet{\mathcal}{OMS}{cmsy}{m}{n}

\newtheorem{theorem}{Theorem}[section]
\newtheorem*{theorem*}{Theorem}
\newtheorem{lemma}[theorem]{Lemma}

\newtheorem{definition}[theorem]{Definition}

\newcommand\numberthis{\addtocounter{equation}{1}\tag{\theequation}}

\makeatletter
\newcommand\RedeclareMathOperator{%
	\@ifstar{\def\rmo@s{m}\rmo@redeclare}{\def\rmo@s{o}\rmo@redeclare}%
}
\newcommand\rmo@redeclare[2]{%
	\begingroup \escapechar\m@ne\xdef\@gtempa{{\string#1}}\endgroup
	\expandafter\@ifundefined\@gtempa
	{\@latex@error{\noexpand#1undefined}\@ehc}%
	\relax
	\expandafter\rmo@declmathop\rmo@s{#1}{#2}}
\newcommand\rmo@declmathop[3]{%
	\DeclareRobustCommand{#2}{\qopname\newmcodes@#1{#3}}%
}
\@onlypreamble\RedeclareMathOperator
\makeatother

\newcommand{\de}{\partial}

\newcommand{\lbar}{\overline}
\newcommand{\td}{\widetilde}
\newcommand{\bsl}{\backslash}

\newcommand{\wlim}{\overset{w}{\rightharpoonup}}

\newcommand{\vf}{\varphi}
\newcommand{\ph}{\phi}
\newcommand{\oo}{\infty}

\newcommand{\om}{\Omega}

\newcommand{\dt}{\delta}

\newcommand{\A}{\mathcal A}

\newcommand{\C}{\mathcal C}

\newcommand{\E}{\mathcal E}
\renewcommand{\H}{\mathcal H}
\renewcommand{\L}{\mathcal L}

\newcommand{\N}{\ensuremath{{\mathcal N}}}

\newcommand{\R}{\Bbb R}

\DeclareMathOperator{\Per}{Per}
\DeclareMathOperator{\im}{im}
\DeclareMathOperator{\dist}{dist}
\DeclareMathOperator{\Det}{Det}
\DeclareMathOperator{\cof}{cof}

\RedeclareMathOperator{\div}{div}

\newcommand{\res}[1]{\left.#1\right\vert}
\newcommand{\nm}[1]{\left\lVert #1 \right\rVert}
\newcommand{\Hd}[1]{\ensuremath{{\mathcal H^{#1}}}}
\newcommand{\inn}[2]{\left\langle #1, #2 \right\rangle}

\newcommand{\W}[1]{\textup{W}^{#1}}
\newcommand{\Leb}[1]{\textup{L}^{#1}}

\numberwithin{equation}{section}

\def\XXint#1#2#3{{\setbox0=\hbox{$#1{#2#3}{\int}$ }
		\vcenter{\hbox{$#2#3$ }}\kern-.6\wd0}}

\title{\vspace{-2cm}Sobolev regularity of the inverse for minimizers of the neo-Hookean energy satisfying condition INV}
\author{Panas Kalayanamit\\
\footnotesize{Instituto de Ciencias de la Ingenier\'ia, Universidad de O'Higgins,}\\ 
\footnotesize{Av.\ Bernardo O'Higgins 611, Rancagua, Chile.}}
\date{}

\begin{document}
	
	\maketitle

 \section*{Introduction}
	
 \paragraph{}
 In nonlinear elasticity, the existence of a minimizer of the neo-Hookean energy 
	\begin{equation*}
		E(u) = \int_{\om} |Du(x)|^2 + H(\det Du(x)) \,dx,
	\end{equation*}
 in a suitable, physically sound, subclass of $\W{1,2}(\om;\R^3)$ has long been a problem that attracts attention due to its challenging nature. Here $u \colon \om \to \R^3$ represents a deformation of a neo-Hookean material whose reference configuration is $\om$, an open domain in $\R^3$ (this problem generalizes in a straightforward way to $\W{1,n-1}(\om;\R^n)$, which is what we shall do in this work). One of the main difficulties is that the class of functions where we wish to find a minimizer should not be so big that it contains functions that represent impossible configurations, yet it must not be too small as to exclude physically tenable configurations. The class  must also have sufficiently nice mathematical properties so that we can apply techniques from the calculus of variations. In this paper, we consider a function space $\A_{0}$ and the weak closure of some of its important subclasses that satisfy the above requirements and show that the minimum of $E$ on each of these spaces is attained. Moreover, we show that if a minimizer $u$ of $E$ in any of these classes satisfies condition (INV), then it enjoys extra good properties.
 
 The neo-Hookean energy is the most widely used form of the polyconvex energies
 \[
 I(u) = \int_{\om} W(x,Du(x)) \,dx
 \]
 for hyperelastic materials, where $W$ is the stored-energy density function. The analysis of minimizers of $I$ in suitable function spaces using techniques from the calculus of variations goes back to \cite{Ball1977}, \cite{Sverak1988}, \cite{Mueller1988}, \cite{GiaquintaGiuseppeSoucek1989}, \cite{Maly1993}, \cite{MullerTangYan1994}, \cite{MullerSpector1995}, \cite{SivaloganathanSpector2000}, \cite{ContiDeLellis2003}. It has long been recognized that choosing a function space to minimize $I$ on is a crucial part of the problem, due to the need to respect certain physical properties that are particular to elasticity. One of the main attributes of a real material is that no interpenetration of matter should occur, which translates into the mathematical requirement that our configuration $u$ should be one-to-one (almost everywhere), and that its inverse $u^{-1}$ should exist in a suitable sense and be sufficiently regular. Moreover, compressing the volume of an elastic body to zero or stretching it to be infinitely large should require infinite energy, hence we shall postulate that  $H\colon (0,\oo) \to [0,\oo)$ in \eqref{neohookean_energy} is a convex function satisfying 
 \begin{equation*}
 	\lim_{t\to\oo} \frac{H(t)}{t} = \lim_{s\to 0} H(s) = +\oo.
 \end{equation*}
 By convention, we let $H(t) = \oo$ if $t \le 0$, which implies that we must have $\det Du >0$ a.e., meaning $u$ with finite energy $E(u)<\oo$ cannot be orientation-reversing on any small volume in the reference configuration. For more on local and global invertibility of Sobolev maps in nonlinear elasticity, see \cite{Ball1981}, \cite{HenaoMora-Corral2015}, \cite{BarchiesiHenaoMora-Corral2017}, \cite{Kroemer2020}, \cite{HenaoMora-CorralOliva2021}.

 Another big consideration we must take into account is the formation of \textit{cavities}, which could be an obstacle to the lower semicontinuity of $E$ (and hence to applying the direct method \cite{Ball1984}, \cite{Kristensen2015}). Since a $\W{1,2}(\om;\R^3)$ function for $\om\subset\R^3$ need not be continuous, it is conceivable that a minimizer of $E$ could exhibit some form of discontinuity. Indeed, in \cite{Ball1982}, it was shown (in the axisymmetric case) that there could be an energetically favourable state that is discontinuous even for simple homogeneous Dirichlet boundary data. Since then, cavitation has become a topic of interest, see e.g.\ \cite{Gent1990}, \cite{MullerSpector1995}, \cite{HorganPolignone1995}, \cite{SivaloganathanSpector2000}, \cite{SivaloganathanSpector2000a}, \cite{Henao2009}, \cite{HenaoMora-Corral2010}, \cite{HenaoSerfaty2013}, \cite{Mora-Corral2014}, \cite{HenaoMora-CorralXu2016}, \cite{Canulef-AguilarHenao2019}, \cite{Negron-MarreroSivaloganathan2024}, \cite{KumarLopez-Pamies2021} and references therein. Even though the class of functions used in this paper excludes the formation of cavities, notions from the theory of cavitations are extremely useful in proving that cavities cannot occur, one of which is the condition \textup{(INV)}.
 
 The condition \textup{(INV)}, introduced in \cite{MullerSpector1995} by M\"uller and Spector in the $\W{1,p}$ case with $p>n-1$ (where $n$ is the dimension of the domain), has proved to be a very important notion in elasticity. Intuitively speaking, a Sobolev map $u$ satisfies \textup{(INV)} if the content inside each ball stays inside and the content outside stays outside after being mapped by $u$. This is strongly related to the invertibility of $u$ and is preserved under the weak $\W{1,p}$ limit. Conti and De Lellis \cite{ContiDeLellis2003} extended the notion to $\W{1,2}$ case (for $n=3$) and showed that many properties of \textup{(INV)} for $p>2$ continue to hold, but their generalized condition \textup{(INV)} is no longer preserved by the $\W{1,2}$ weak limit. Indeed, they constructed a sequence of bi-Lipschitz homeomorphisms whose weak limit $u$ in $\W{1,2}$ fails to satisfy \textup{(INV)}, reverses the orientation on a set of positive measure, and $u^{-1}$ is not a $\W{1,1}_{loc}$ function ($u^{-1}$ has a jump discontinuity, which means that $u$ put matters from different parts of the body into contact with one another, violating the non-interpenetration of matter). Hence, the failure if \textup{(INV)} in the limit does lead to pathological behaviours we wish to avoid. However, in \cite{DolezalovaHenclMaly2023} Dole\v{z}alov\'a, Hencl, and Mal\'y showed that if we restrict our attention to maps that can be obtained as the limit of a sequence of maps $(u_j)_j$ that: (i) are homeomorphisms; and (ii) satisfy $\sup_{j} \int_\om (\det Du_j(x))^{-2}\,dx < \oo$, then condition \textup{(INV)} does pass to the limit. 
 
 In this work, we prove the lower semicontinuity of $E$ on the weak closure $\lbar{\A}$ of any subclasses $\A$ of $\A_0$, the class of functions satisfying the divergence identities (defined in Theorem \ref{main_theorem_0}; this is closely related to the \textit{surface energy} introduced in \cite{HenaoMora-Corral2010} by Henao and Mora-Corral), where all members of $\lbar{\A}$ have positive Jacobian, such as $\lbar{\A_{dif}}$ and $\lbar{\A_{hom}}$ (defined in Theorem \ref{main_theorem_1}). In a recent work \cite{DolezalovaHenclMolochanova2023} Dole\v{z}alov\'a, Hencl, and Molchanova showed that, under similar assumptions regarding the neo-Hookean energy as in \cite{DolezalovaHenclMaly2023}, $E$ is lower semicontinuous on $\lbar{\A_{hom}}$ and that $E$ admits a minimizer there. Assuming the same type of coercivity as in \cite{DolezalovaHenclMaly2023, DolezalovaHenclMolochanova2023}, we show that minimizers of $u$ in $\lbar{\A_{hom}}$ not only satisfy (INV) but also the divergence identities \eqref{div_identity}, $\Det Du = (\det Du) \L^3$ and its inverse $u^{-1}$ has $\W{1,1}$ regularity. On the other hand, we take the opportunity for a slight generalization: the lower semicontinuity of $E$, and hence the existence of a minimizer, still holds even if their coercivity condition on $H(\det Du)$ is removed.

 \section{Preliminaries}
	
 \paragraph{}
 	In this article, we will work with a thick Dirichlet boundary condition $b$ that can be extended to a homeomorphism on a slightly larger domain. Equivalently, we let $\td\om, \om \subset \R^n$ be open Lipschitz domains such that $\td\om \subset\subset \om$ (in particular, we assume that $\td\om$ and $\om$ have finite perimeters) and let $b\colon \om \to \R^n$ be a fixed $\W{1,n-1}$ orientation-preserving homeomorphism\footnote{By orientation-preserving, we mean that $\det Db > 0$ a.e. There is a connection between a Sobolev homeomorphism being orientation-preserving and being \textit{sense-preserving}, i.e.\ that the later implies the former when $b \in \W{1,n-1}(\om;\R^n)$ (see  \cite{HenclMaly2010}).}  onto its range that satisfies the divergence identity \eqref{div_identity} (this is the case if $b$ is a bi-Lipschitz orientation-preserving homeomorphism, for example). We shall write
	\[
	\td\om_b := b(\td\om)\ \text{ and }\ \om_b := b(\om).
	\]
	We will require that the Sobolev functions $u$ in our class of interest satisfy $u = b$ on $\om\bsl\td\om$.	The energy functional used throughout this article is the ($n-1$) neo-Hookean energy
	\begin{equation}\label{neohookean_energy}
		E(u) := \int_{\om} |Du(x)|^{n-1} + H(\det Du(x)) \,dx,
	\end{equation}
	where $H\colon (0,\oo) \to [0,\oo)$ is a convex function satisfying 
	\begin{equation}\label{condition_H0}
		\lim_{t\to\oo} \frac{H(t)}{t} = \lim_{s\to 0} H(s) = +\oo.
	\end{equation}

	\paragraph{}
	For the definition and properties of the approximate derivative of a measurable function $u$ defined on (a subset of) $\R^n$, we refer to \cite{Federer1996_GMT}, \cite{EvansGariepy1992_1st} or \cite{MullerSpector1995}. We shall use, in particular, that if $u$ is approximately differentiable at $x_0$, then $u$ is defined and is approximately continuous at $x_0$. It is well-known that a Sobolev function $u$ is approximately differentiable a.e., and that its approximate differential coincides a.e.\ with its distributional derivative $Du$. Note that we do not identify functions that coincide almost everywhere.	Following \cite{HenaoMora-Corral2012}, we shall give the following definition.

 \begin{definition}\label{defn_geometric_image}
 	For a Sobolev function $u \in \W{1,n-1}(\om;\R^n)$ such that $\det Du > 0$ a.e., we define the \emph{geometric image} of a measurable set $E \subset \R^n$ under $u$ to be
 	\[
 	\im_G(u, E) := u(E \cap \om_0),
 	\]
 	where $\om_0$ is the set of points $x \in \om$ that satisfy:
 	\begin{itemize}
 		\item[(a)] the approximate differential of $u$ at $x$ exists and equals $Du(x)$,
 		
 		\item[(b)] $\det Du(x) > 0$, and
 		
 		\item[(c)] there exists $w \in C^1(\R^n;\R^n)$ and a compact set $K$ of density $1$ at $x$ such that $\res{u}_K = \res{w}_K$ and $\res{Du}_K = \res{Dw}_K$.
 	\end{itemize}
 \end{definition}
 	
 	This definition of the geometric image is slightly different that the one used in \cite{MullerSpector1995} and \cite{ContiDeLellis2003}, however $\om_0$ is of full measure in $\om$, so our definition of the geometric image is equivalent to the other definition up to an $\L^n-$null set (see \cite[Lemma 2.5]{HenaoMora-Corral2012}). 
	
	An important property of the geometric image is that whenever $u \in \W{1,n-1}(\om;\R^n)$ with $\det Du > 0$ a.e.\ is also one-to-one a.e., we have $\res{u}_{\om_0}$ is one-to-one (see \cite[Lemma 3]{HenaoMora-Corral2011}) and the \textit{change of variables formula} (or the \textit{area formula}) 
	\[
	\int_E (\vf\circ u)(x) \det Du(x) \,dx = \int_{\im_G(u,E)} \vf(y) \,dy 
	\]
	holds, where $E\subset \om$ is a measurable set and $\vf\colon \R^n \to \R$ is a measurable function (see, for example, \cite[Proposition 2.6]{MullerSpector1995}). For any such $u$, we may define its \textit{inverse} to be the function $u^{-1} \colon \im_G(u,\om) \to \R^n$ mapping $y$ to the unique point $x\in\om_0$ such that $u(x) = y$. The change of variables formula takes the more general form
	\begin{equation}\label{eqn_change_of_variables_formula_with_multiplicity}
		\int_E (\vf\circ u)(x) |\det Du(x)| \,dx = \int_{\R^n} \vf(y)\, \N_E(y) \,dy,
	\end{equation}
	where
	\[
	\N_E(y) := \# \{ x \in E : u \textup{ is approximately differentiable at $x$ and } u(x) = y \},
	\]
	if we only assume that $u\colon\om\to\R^n$ is approximately differentiable a.e. and whenever either integral exists (see \cite{Federer1996_GMT} or \cite{MullerSpector1995}). Here $\# A$ denotes the cardinality of $A$.

 	\begin{definition}\label{defn_degree_and_im_T}
		Let $u\in\W{1,n-1} \cap \Leb{\oo}(\om;\R^n)$ and $U\subset\subset \R^n$ be a \emph{good open set}\footnote{An open set $U\subset \om$ is called a good open set with respect to $u$ if $U$ is compactly contained in $\om$, has a $C^2$ boundary, satisfies
			\begin{itemize}
				\item $\res{u}_{\de U} \in \W{1,n-1} \cap \Leb{\oo}(\de U;\R^n)$ and $\res{(\cof Du)}_{\de U} \in \Leb{1}(\de U)$,
				
				\item $\Hd{n-1}(\de U\bsl \om_0) = 0$ ($\om_0$ is defined as in Definition \ref{defn_geometric_image}) and $D(\res{u}_{\de U})(x)$ coincides with the orthogonal projection of $Du(x)$ onto $T_x(\de U)$, the tangent space of $\de U$ at $x$,
			\end{itemize}
			and a few extra properties. The full definition of $\mathcal U_u$, the family of the good open sets in $\om$, can be found in \cite[Definition 2.15]{HenaoMora-Corral2012}. } with respect to $u$. We may define the \emph{degree} $\deg(u,\de U, \cdot)$ to be the unique $\Leb{1}(\R^n)$ function\footnote{See \cite[Remark 3.3]{ContiDeLellis2003} to see that the degree is well-defined (see also \cite[Proposition 2.10]{HenaoMora-Corral2012}).  } such that 
		\begin{equation}\label{degree_defn}
			\int_{\R^n} \div g(y) \deg(u,\de U, y) \,dy = \int_{\de U} g(u(x)) \cdot (\cof Du(x))[\nu(x)] \,d\Hd{n-1}(x)
		\end{equation}
		for each vector field $g\in C^1(\R^n;\R^n)$. $\deg(u,\de U, \cdot)$ is integer-valued a.e. The \emph{ topological image} of $u$, $\im_T(u, U)$, is the set of points where 
		\[
		A_{u,U} := \{ y \in \R^n : \deg(u,\de U, y) \ne 0 \}
		\]
		has density $1$, i.e.\ 
		\[
		\im_T(u, U) := \left\{ y \in \R^n : \lim_{r\searrow 0} \frac{\L^n(B(y,r) \cap A_{u,U})}{\L^n(B(y,r))} = 1. \right\}
		\]
	\end{definition}

	\begin{definition}
		A function $u\in\W{1,n-1} \cap \Leb{\oo}(\om;\R^n)$, is said to satisfy \emph{condition} \textup{(INV)} if for every $x_0\in \om$, the following conditions hold for a.e.\ $r\in (0,\dist(x_0,\de \om))$.
		\begin{itemize}
			\item[(i)] $u(x) \in \im_T(u,B(x_0,r))$ for a.e. $x \in B(x_0,r)$.

			\item[(ii)] $u(x) \notin \im_T(u,B(x_0,r))$ for a.e. $x \in \om\bsl B(x_0,r)$.
		\end{itemize}
	\end{definition}
	Following \cite{ContiDeLellis2003}, the \textit{topological image of a point} $x\in\R^n$ is defined as 
	\[
	\im_T(u,\{x\}) := \bigcap_{r\in R_x} \im_T(u,B(x,r)),
	\]
	where $R_x$ is the set of \textit{good radii}\footnote{See \cite[Definition 3.11, Definition 3.13 and Remark 4.4]{ContiDeLellis2003} for the definition of good radii.}, which is a subset of full measure in $(0,\dist(x,\de \om))$. An equivalent definition is given in \cite[Definition 2.18]{HenaoMora-Corral2012} when $u$ satisfies condition (INV).

 Based on the idea presented in \cite{GiaquintaGiuseppeSoucek1989} and \cite{Mueller1988}, we work with maps satisfying the divergence identities.
	\begin{definition}
		A Sobolev function $u\in \W{1,n-1}(\om;\R^n)$ with $\det Du \in\Leb{1}_{loc}(\om)$ is said to satisfy the \textit{divergence identities} if
		\begin{equation}\label{div_identity}
			\int_{\om} (\div g)(u(x)) \phi(x) \det Du(x) \,dx = - \int_{\om} g(u(x)) \cdot (\cof Du(x))[D\phi(x)] \,dx
		\end{equation}
		for every $\phi\in C^1_c(\om)$ and every $g\in C^1_c(\R^n;\R^n)$.
	\end{definition}
	See \cite{BarchiesiHenaoMora-CorralRodiac2023_relaxation} for a more recent work that employs this notion. We will show below (in Lemma \ref{lemma_1}) that, for a.e.\ injective maps, the divergence identities are equivalent to condition \textup{(INV)} together with $\Det Du = (\det Du)\L^n$. 
	
	There are a few ways to formulate the divergence identities in the literature. In particular, they are equivalent to the statements that the \textit{surface energy} of $u$ is zero, i.e.\ $\E(u)=0$ (or $\lbar\E(u)=0$, see \cite[Section 4]{HenaoMora-Corral2010} and also \cite[Corollary 1.6.5]{Llavona1986_approx_book}), where
	\[
	\E(u) := \sup_{\nm{f}_\oo \le 1} \int_{\om} D_x f(x,u(x)) \cdot \cof Du(x) + \div_y\, f(x,u(x)) \det Du(x) \,dx,
	\]
	and the supremum is taken over all $f \in C^1_c(\om \times \R^n; \R^n)$, whereas $\lbar\E$ is defined analogously but restricting $f$ to be of the form $f(x,y) = \phi(x)g(y)$ for $\phi\in C^1_c(\om)$ and $g\in C^1_c(\R^n;\R^n)$. Note that by taking $g(y) = \frac{1}{n}y$, the expression \eqref{div_identity} can be used to define the \textit{distributional determinant} of $Du$:
	\[
	\inn{\Det Du}{\phi} := -\frac{1}{n} \int_{\om} u(x) \cdot (\cof Du(x))[D\phi(x)] \,dx
	\]
	whenever $u \in \W{1,n-1}\cap \Leb{\oo}(\om;\R^n)$. When $\Det Du$ can be identified with a Radon measure $\mu$, we shall write $\Det Du(E)$ for a measurable set $E$ to denote $\mu(E)$.

 \section{Main results}

 \begin{lemma}\label{lemma_1}
 	Let $u\in \W{1,n-1}(\om;\R^n)$ and suppose that $u=b$ in $\om\bsl\td\om$, $u$ is one-to-one a.e., $\det Du \in \Leb{1}_{loc}(\om)$, $\det Du > 0$ a.e.\ Then the following are equivalent.
 	\begin{itemize}
 		\item[(i)] $u$ satisfies the divergence identities \eqref{div_identity}.
 		
 		\item[(ii)] $u^{-1} \in \W{1,1}(\om_b ; \R^n)$ and $\im_G(u,\om) = \om_b$ a.e.
 		
 		\item[(iii)] $u\in \Leb{\oo}(\om;\R^n)$, $u$ satisfies condition \textup{(INV)} and $\Det Du = (\det Du)\L^n$.
 	\end{itemize}
 \end{lemma}
 
 \textit{Remark}: A priori, we only require that $\det Du \in \Leb{1}_{loc}(\om)$, but whenever any of the equivalent statements hold, we actually have $\det Du \in \Leb{1}(\om)$. Indeed, this immediately follows from (ii) and the change of variable formula
 \[
 \int_\om \det Du(x) \,dx = \int_{\im_G(u,\om)}1 \,dy = |\om_b| < \oo.
 \]
 
 \begin{proof}
 	The implication $(ii) \implies (i)$ is proved in detail in \cite{BarchiesiHenaoMora-CorralRodiac2023_relaxation} for the case $n=3$, but their technique applies to any $u \in \W{1,n-1}(\om;\R^n)$ (the converse of this is also sketched there, but for the convenience of the readers, we shall give more details of the proof down below).

 	Assuming $(i)$, we shall first prove that $\im_G(u,\om) = \om_b$ a.e.\ and, in particular, $u\in\Leb{\oo}(\om,\R^n)$. Let $U$ be an arbitrary good open set such that $\td\om \subset\subset U \subset\subset \om$.\footnote{Although we have not shown that $u\in\Leb{\oo}(\om;\R^n)$ yet at this stage, we know that $u\in\Leb{\oo}(\om\bsl\td\om;\R^n)$ since $u$ coincides with $b$ on $\om\bsl\td\om$, hence the argument in \cite[Lemma 2.16]{HenaoMora-Corral2012} can be applied to an arbitrary open set $\td U \subset \om\bsl\td\om$ with $C^2$ boundary to prove the existence of a good open set $U$. Thus $\deg(u,\de U,\cdot)$ and the topological image $\im_T(u,U)$ are well-defined} We shall follow closely the proof of \cite[Theorem 4.1]{BarchiesiHenaoMora-Corral2017}\footnote{The result can be adapted to the case $p=n-1$ too, even though it is stated for $p>n-1$ in \cite{BarchiesiHenaoMora-Corral2017}. This is because the proof relies only on the fact that the degree satisfies the relation \eqref{degree_defn} and not on the continuity of $u$ on $\de U$. The  necessary change is that $\deg(u,\de U,y)$ in the definition of $\im_T(u,U)$ must be understood in terms of \eqref{degree_defn} instead of the classical Brouwer degree (they coincide a.e.\ and for almost every $U$ whenever $\res{u}_{\de U}\in C(\de U;\R^{n-1})$). } and consider a family $\{ \ph_\dt \}_{\dt>0}\subset C^1_c(\om)$ such that $\ph_\dt \nearrow \chi_U$ pointwise on $\om$ as $\dt \searrow 0$ and 
 	\[
 	\lim_{\dt\to 0} \int_\om g(u(x)) \cdot (\cof Du(x))[D\ph_\dt(x)] \,dx = - \int_{\de U} g(u(x)) \cdot (\cof Du(x))[\nu(x)]  \,d\Hd{n-1}(x)
 	\]
 	for any arbitrary $g\in C^1_c(\R^n;\R^n)$, where $\nu(x)$ is the outward unit normal at $x\in \de U$ (see the proof of \cite[Theorem 2]{HenaoMora-Corral2010} for the existence of $\{ \ph_\dt \}_{\dt>0}$). Thus, we have
 	\begin{align}
 		\int_{\R^n} \div g(y) \deg(u, \de U, y) \,dy &= \int_{\de U} g(u(x)) \cdot (\cof Du(x))[\nu(x)] \,d\H^{n-1}(x) \nonumber \\
 		&= - \lim_{\dt\to 0} \int_\om g(u(x)) \cdot (\cof Du(x))[D\ph_\dt(x)] \,dx \nonumber \\
 		&= \lim_{\dt\to 0} \int_\om    (\div g)(u(x)) \ph_\dt(x) \det Du(x) \,dx \label{lemma_1_eqn1} \\
 		&= \int_U (\div g)(u(x)) \det Du(x) \,dx \label{lemma_1_eqn2} \\
 		&= \int_{\R^n} \div g(y)\, \N_U(y) \,dy, \nonumber
 	\end{align}
 	where we used the divergence identities in \eqref{lemma_1_eqn1}, the dominated convergence theorem in \eqref{lemma_1_eqn2}, and the change of variables formula \eqref{eqn_change_of_variables_formula_with_multiplicity} for the last equality. This shows that $\deg(u, \de U, \cdot) - \N_U = c$ a.e.\ for some constant $c \in \Bbb Z$ (since both functions are integer-valued). We may then argue as in \cite[Theorem 4.1]{BarchiesiHenaoMora-Corral2017} to conclude that
 	\[
 	\deg(u, \de U, \cdot) = \N_U \  \text{a.e.}\quad\text{and}\quad  \im_T(u,U) = \im_G(u,U) \  \text{a.e.}\footnote{see \cite[equation (5.1)]{BarchiesiHenaoMora-Corral2017}}  
 	\]
 	Since $b$ is a homeomorphism and $u$ coincides with $b$ on $\de U$, we have $\im_T(u,U) = b(U)$. Then, since $u$ is one-to-one a.e. and $u$ coincides with $b$ on $\om\bsl U$, 
 	\[
 	im_G(u,\om) = im_G(u,U) \cup im_G(u,\om\bsl U) \overset{\text{a.e.}}{=} b(U) \cup b(\om\bsl U) = b(\om) = \om_b.
 	\]

 	$(i)\iff (iii)$: Suppose that $u$ satisfies the divergence identities, we have seen that $\im_G(u,\om) = \om_b$ a.e. Hence, in particular, $u \in \Leb{\oo}(\om;\R^n)$. To show that $u$ satisfies \textup{(INV)}, we consider a fixed $x_0 \in \om$. Using \cite[Theorem 4.1]{BarchiesiHenaoMora-Corral2017}, it follows that for a.e.\ $r \in (0, \dist(x_0,\de\om))$, since each $B(x_0,r)$ is a good open set in the sense of that theorem, we have  
 	\[
 	\deg(u, \de B(x_0,r) , \cdot) = \N_{B(x_0,r)} \  \text{ a.e.}
 	\]
 	Hence, as in \cite[Lemma 5.1]{BarchiesiHenaoMora-Corral2017}, $u(x) \in A_{u,B(x_0,r)}$ (defined as in Definition \ref{defn_degree_and_im_T}) for a.e.\ $x\in B(x_0,r)$. By applying Lebesgue differentiation theorem to $\chi_{A_{u,B(x_0,r)}}$ and using that $\det Du > 0$ a.e. and Federer's change of variable formula (see, e.g., \cite[Proposition 2.3 and Lemma 2.5]{HenaoMora-Corral2012}), we find that $u(x) \in \im_T(u,B(x_0,r))$ for a.e.\ $x \in B(x_0,r)$ as wanted. The proof that $u(x) \notin \im_T(u,B(x_0,r))$ for a.e. $x \in \om\bsl B(x_0,r)$ is similar, noting that $\N_{B(x_0,r)}(u(x)) = 0$ for a.e.\ $x\in \om\bsl B(x_0,r)$ since $u$ is one-to-one a.e. Hence condition \textup{(INV)} is satisfied. The equivalence of $(i)$ and $(iii)$ then follows from \cite[Lemma 5.3]{HenaoMora-Corral2012}.

 	$(i)\implies(ii)$: We have already seen that $(i)$ implies $\im_G(u,\om) = \om_b$ a.e., so we only need to prove the Sobolev regularity of $u^{-1}$. We shall follow the argument made in \cite{BarchiesiHenaoMora-CorralRodiac2023_relaxation} (see \cite[Proposition 4.12]{BarchiesiHenaoMora-CorralRodiac2023_harmonic} for the axisymmetric case). Indeed, since we have shown that $(i)$ implies condition (INV), we may apply \cite[Theorem 3.4]{HenaoMora-Corral2015} (note also that $\td u$ that appears in the theorem coincides with $u$ since we already proved that $(i)$ implies $\Det Du = (\det Du)\L^n$, i.e.\ $u$ opens no cavities) to a good open set $U$ such that $\td\om \subset\subset U \subset\subset \om$ and see that
 	\[
 	Du^{-1}\llcorner b(U) =  (Du \circ u^{-1})^{-1}\L^n \llcorner b(U).
 	\]
 	(here $Du^{-1}$ on the left hand side denotes the distributional derivative of $u^{-1}$, which, a priori, is not necessarily a function). In particular, this shows that $Du^{-1} \in \Leb{1}(b(U); \R^{n\times n})$ due to the change of variable formula
 	\[
 	\int_{b(U)} (Du \circ u^{-1}(y))^{-1} \,dy = \int_U (Du(x))^{-1} |\det(Du(x))| \,dx = \int_U |\cof Du(x)| \,dx
 	\]
 	and the fact that $u\in\W{1,n-1}(\om; \R^n)$. Since $Du^{-1}$ is integrable on two overlapping open sets $b(U)$ and $\om_b\bsl\lbar{\td{\om}_b}$, it follows that $Du^{-1} \in \Leb{1}(\om_b; \R^{n\times n})$ and we are done.
 \end{proof}

 \begin{lemma}\label{lemma_2}
 	Let $(u_j)_j$ be a sequence in $\W{1,n-1}(\om;\R^n)$ such that $u_j=b$ in $\om\bsl\td\om$. Assume further that
 	\begin{itemize}
 		\item[(i)] $u_j$ is one-to-one a.e., $\det Du_j > 0$ a.e.,
 		
 		\item[(ii)] $u_j$ satisfies the divergence identities \eqref{div_identity} for each $j\in\Bbb N$ and
 		
 		\item[(iii)] $\sup_j E(u_j) < \oo$.
 	\end{itemize}
 	Suppose that  and there exist  $u \in \W{1,n-1}(\om;\R^n)$ with $\det Du \ge 0$ a.e.\ such that $u_j \to u$ a.e. Then 
 	\begin{itemize}
 		\item[(a)] $u$ is one-to-one a.e., $\det Du > 0$ a.e.,
 		
 		\item[(b)] $\im_G(u,\om) = \om_b$ a.e. and
 		
 		\item[(c)] $\det Du_j \wlim \det Du$ in $\Leb{1}$.
 	\end{itemize}
 \end{lemma}

 \begin{proof}
 	From the assumption $(iii)$ that the energy $E(u_j)$ is uniformly bounded and \eqref{condition_H0}, $(\det Du_j)_j$ is equi-integrable by the de la Vall\'ee Poussin criterion. A general argument  (see, for example, the proof of \cite[Theorem 5.1]{MullerSpector1995}) shows that the weak $\Leb{1}$ limit (after passing to a subsequence) of $(\det Du_j)_j$ exists and is (strictly) positive a.e. Hence, according to \cite[Theorem 2, (ii)]{HenaoMora-Corral2010}, $u$ is one-to-one a.e.\ and 
 	\[
 	\det Du_j \wlim |\det Du | \quad \text{in}\ \Leb{1},
 	\]
 	Since we assume that $\det Du \ge 0$ a.e., $(a)$ and $(c)$ are proved.
 	
 	To prove $(b)$, we will now show that 
 	\begin{equation*}\label{eqn_0}
 		\chi_{\im_G(u,\om)} = \chi_{\om_b} \  \text{a.e.}
 	\end{equation*}
 	According to Lemma \ref{lemma_1}, we know that $\chi_{\im_G(u_j,\om)} = \chi_{\om_b}$ for each $j$. Thus, for any $\psi\in C_c(\R^n)$, by Egorov's theorem and $\det Du_j \wlim \det Du $ from $(c)$, we have
 	\begin{align*}
 		\lim_{j\to\oo} \int_{\R^n} \chi_{\om_b}(y) \psi(y) \,dy &= \lim_{j\to\oo} \int_{\om} \psi(u_j(x)) \det Du_j(x) \,dx \\
 		&= \int_{\om} \psi(u(x)) \det Du(x) \,dx \\
 		&= \int_{\R^n} \chi_{\im_G(u,\om)}(y) \psi(y) \,dy,
 	\end{align*}
 	and hence $\chi_{\im_G(u,\om)} = \chi_{\om_b}$ a.e. (this is also a consequence of \cite[Theorem 2 (iii)]{HenaoMora-Corral2010}, but, for convenience of the readers, we included here a short proof).
 \end{proof}

 \paragraph{}
 Let $\A$ be any subset of $\W{1,n-1}(\om;\R^n)$. We shall denote the (sequential) weak $\W{1,n-1}$ closure of $\A$ by
 \[
 \lbar{\A} := \Big\{ u \in \W{1,n-1}(\om;\R^n) : \textup{There are } u_j \in \A \textup{ such that } u_j \wlim u \Big\}.
 \]
 Note that $\lbar{\A}$ is sequentially closed, i.e.\ if $u_j \in \lbar{\A}$ and $u_j \wlim u$ in $\W{1,n-1}$, then $u \in \lbar{\A}$. This is due to the fact that weakly converging sequences are bounded, and that the weak topology on a norm-bounded set in $\W{1,n-1}(\om;\R^n)$, a reflexive and separable Banach space, is metrizable (see also \cite[Theorem 5.1]{DolezalovaHenclMolochanova2023} for a proof using a diagonalization argument for a similar statement).

 \paragraph{}
 From this point onward, we shall assume that $\td\om_b$ and $\om_b$ are Lipschitz domains (so, in particular, they have finite perimeters). Note that this is automatically the case if $b$ is a bi-Lipschitz orientation-preserving homeomorphism.
 
 \begin{theorem}\label{main_theorem_0}
 	Let $E$ be an energy functional of the form \eqref{neohookean_energy}, where $H$ satisfies \eqref{condition_H0}. Define
 	\begin{align*}
 		\A_0 := \Big\{ u \in \W{1,n-1}(\om &; \R^n) : u \ \textup{is one-to-one a.e., } \det Du > 0 \ \textup{ a.e., } u = b \ \textup{on }\om\bsl\td\om,\\
 		&u \ \textup{satisfies the divergence identities \eqref{div_identity} and}\  E(u) \le E(b)\Big\},
 	\end{align*}
 	and let $\A$ be a subclass of $\A_0$ such that $\det Du \ge 0$ a.e. for any $u\in \lbar\A$. Then 
 	\begin{itemize}
 		\item[(i)] $E$ is sequentially weakly lower semicontinuous on $\lbar{\A}$.
 		
 		\item[(ii)] $\begin{aligned}[t]
 			\lbar{\A} \subset \Big\{ u \in \W{1,n-1}(\om &;\R^n) :  u \ \textup{is one-to-one a.e., } \det Du > 0 \ \textup{ a.e., } \\
 			& u = b \ \textup{on }\om\bsl\td\om,\ \im_G(u,\om) = \om_b\ \text{ a.e., and}\  E(u) \le E(b)\Big\}.
 		\end{aligned}$
 		
 		\item[(iii)] If $u \in \lbar{\A}$ satisfies condition \textup{(INV)}, then $u$ has all of the equivalent properties stated in Lemma \ref{lemma_1}.
 	\end{itemize}
 \end{theorem}

 \begin{proof}
 	$(i)$: Let $(u_j)_j$ be a sequence in $\lbar{\A}$, i.e.\ each $u_j$ is a limit of some sequence in $\A$. By Lemma \ref{lemma_2}, each $u_j$ is one-to-one a.e., satisfies $\det Du_j > 0$ a.e.\ and $\im_G(u_j,\om) = \om_b$ a.e.\ for all $j\in\Bbb N$. Suppose that $u_j \wlim u$, from  \eqref{condition_H0}, $(\det Du_j)_j$ is weakly compact in $\Leb{1}$, so for any arbitrary subsequence of $(u_j)_j$, we may extract a further subsequence $(u_{j_k})_k$ such that $\det Du_{j_k} \wlim \theta$ for some $\theta \in \Leb{1}(\om;\R^{n\times n})$ and $u_{j_k} \to u$ a.e.
 	
 	Since $u$ also belongs to $\lbar{\A}$ (see the discussion prior to this theorem), Lemma \ref{lemma_2} implies that $u$ is one-to-one a.e.\ and  $\det Du > 0$ a.e. According to \cite[Lemma 4.1]{MullerSpector1995}\footnote{As $\im_G(u_j,\om)$ is a fixed set for all $j\in\Bbb N$, we trivially have $\chi_{\im_G(u_j,\om)} \to \chi_{\om_b}$ in $\Leb{1}$. The fact that $\theta > 0$ a.e.\ can be argued the same way as in Lemma \ref{lemma_2}.}, we see that
 	\[
 	\theta = |\det Du| = \det Du \quad \text{a.e.}
  	\]
 	This shows that $\det Du_{j_k} \wlim \det Du$, and thus also for the whole sequence, i.e.\ $\det Du_{j} \wlim \det Du$. We have shown that the Jacobian is weakly continuous in the class $\lbar{\A}$, therefore $E$ is weakly lower semicontinuous on $\lbar{\A}$ as a consequence of the classical result \cite[Theorem 5.4]{BallCurrieOlver1981} by Ball-Currie-Olver.

 	$(ii)$:  Let $u\in\lbar{\A}$. From Lemma \ref{lemma_2}, we know that $u$ is one-to-one a.e.,  $\det Du > 0$ a.e., and $\im_G(u,\om) = \om_b$ a.e. The fact that $u=b$ on $\om\bsl\td\om$ is obvious. The inequality $E(u) \le E(b)$ is the consequence of the sequential weak lower semicontinuity of $E$ that was proved in $(i)$.

 	$(iii)$: Suppose that $u \in \lbar{\A}$ satisfies condition \textup{(INV)}. From $(ii)$, we know that $\im_G(u,\om) = \om_b$ a.e., thus
 	\begin{equation}\label{eqn_1}
 		\Per(\im_G(u,\om)) = \Per(\om_b) < \oo.
 	\end{equation}
 	By \cite[Theorem 4.2]{ContiDeLellis2003} (see also \cite{DeLellisGhiraldin2010}), it follows that $\Det Du$ is a Radon measure and
 	\begin{equation}\label{eqn_2}
 		\Det  Du = (\det Du)\L^n + \sum_{\xi \in \C} \L^n(\im_T(u,\{\xi\}))\, \dt_\xi,
 	\end{equation}
 	for some (at most) countable set $\C$, the set of points in which $u$ opens up cavities, that is contained in $\om$. Here $\dt_\xi$ denotes the Dirac mass at $\xi$. In fact, from the assumption that $b$ satisfies the divergence identities, Lemma \ref{lemma_1} (iii) says that $b$ cannot open any cavity, thus we must have $\C \subset \lbar{\td\om}$ since $u$ coincides with $b$ on $\om\bsl\td\om$. 
 	
 	We will now show that all the cavities must have zero volume. Indeed, taking an arbitrary good open set $U$ such that $\td\om\subset\subset U \subset\subset \om$, 
 	\begin{align*}
 		\L^n(\om_b) &= \L^n(\om_b \bsl b(U)) +  \int_{U} \det Db(x) \,dx \\
 		&= \L^n(\om_b \bsl b(U)) + \Det Db(U) \\
 		&= \L^n(b(\om \bsl U)) + \Det Du(U), \numberthis \label{eqn_3}
 	\end{align*}
 	where the last equality follows from \cite[Proposition 2.17 (iv)-(v)]{HenaoMora-Corral2012}, i.e.\ 
 	\[
 	\Det Db(U) = \L^n(\im_T(b,U)) = \L^n(\im_T(u,U)) = \Det Du(U)
 	\]
 	because $\res{u}_{\de U} = \res{b}_{\de U}$. Recalling that $b$ and $u$ agree on $\om \bsl U$, we substitute that into \eqref{eqn_3} to get
 	\begin{align*}
 		\L^n(\om_b) &= \L^n(\im_G(u,\om\bsl U)) + \int_{U} \det Du(x) \,dx + \sum_{\xi \in \C} \L^n(\im_T(u,\{\xi\})) \\
 		&= \L^n(\im_G(u,\om)) + \sum_{\xi \in \C} \L^n(\im_T(u,\{\xi\})) \\
 		&= \L^n(\om_b) + \sum_{\xi \in \C} \L^n(\im_T(u,\{\xi\})), 
 	\end{align*}
 	where we used the fact that $\C \subset U$ in the first equality, thus 
 	\[
 	\sum_{\xi \in \C} \L^n(\im_T(u,\{\xi\})) = 0.
 	\]
 	
 	The above equality and \eqref{eqn_2} show that $\Det  Du = (\det Du)\L^n$. Since we already know that $u$ satisfies \textup{(INV)}, we have fulfilled condition $(iii)$ in Lemma \ref{lemma_1}, and hence the rest of the equivalent conditions in Lemma \ref{lemma_1} also hold.
 \end{proof}

 While it is plausible that $\det Du > 0$ a.e.\ for every $u \in \lbar{\A_0}$ (in particular, we expect that the assumption $\det Du \ge 0$ a.e.\ in Lemma \ref{lemma_2} can be dropped), we don't have the proof of this statement yet, hence it is not clear if $E$ is weakly lower semicontinuous on the class $\lbar{\A_0}$ or not. However, a remarkable result by Hencl and Onninen in \cite{HenclOnninen2018} implies that if each member of the sequence $(u_j)_j$ is a $\W{1,n-1}$ homeomorphism whose Jacobian is non-negative  a.e., then $u_j \wlim u$ in $\W{1,n-1}$ is enough to guarantee that $\det Du \ge 0$ a.e. This gives us the following existence theorem.

  \begin{theorem}\label{main_theorem_1}
  	Let $b\in\W{1,n-1}(\om;\R^n)$ be an orientation-preserving homeomorphism satisfying the divergence identities \eqref{div_identity} and let $E$ be an energy functional of the form \eqref{neohookean_energy}, where $H$ satisfies \eqref{condition_H0}. For any nonempty subclass $\A \subset \A_0$ such that $\det Du \ge 0$ a.e.\ for every $u\in \lbar\A$, the energy functional $E$ admits a minimizer in $\lbar\A$. In particular, let
  	\begin{itemize}
  		\item[] $\begin{aligned}[t]
  			\A_{dif} := \Big\{ u \in \W{1,n-1}(\om;\R^n) :  u \ &\textup{is a\, $C^1$ diffeomorphism,}\ \det Du > 0 \ \textup{ a.e., } \\
  			&\quad\quad u = b \ \textup{on }\om\bsl\td\om \text{ and}\  E(u) \le E(b)\Big\}, \text{ and}
  		\end{aligned}$
  		
  		\item[] $\begin{aligned}[t]
  			\A_{hom} :=  \Big\{ &u \in \W{1,n-1}(\om;\R^n) : u \ \ \textup{is a homeomorphism,}\ \det Du > 0 \ \textup{a.e.,}\\
  			&u\ \text{satisfies Lusin's condition \textup{(N)},}\ u = b \ \textup{on }\om\bsl\td\om \text{ and }  E(u) \le E(b) \Big\},
  		\end{aligned}$
  	\end{itemize}
  	then $E$ has a minimizer in $\lbar{\A_{hom}}$ if $b$ is a homeomorphism that satisfies Lusin's condition \textup{(N)}, and $E$ has a minimizer in $\lbar{\A_{dif}}$ if $b$ is a $C^1$ diffeomorphism.
  	
  	Moreover, if a minimizer $u$ of $E$ in any of the above classes satisfies condition \textup{(INV)}, then it also satisfies the divergence identities \eqref{div_identity}, $\Det Du = (\det Du)\L^n$ and $u^{-1} \in \W{1,1}(\om_b;\R^n)$.
  \end{theorem}
  
  \begin{proof}
  	Let $\A$ be a subclass of $\A_0$ such that $\det Du \ge 0$ a.e. for every $u\in \lbar\A$. Due to the bound on the energy $\sup_{u\in\lbar{\A}} E(u) \le E(b)$, $\lbar{\A}$ is weakly compact in $\W{1,n-1}(\om;\R^n)$, hence the sequential weak lower semicontinuity of $E$ implies that $E$ admits a minimizer in $\lbar{\A}$. A minimizer $u$ has the properties stated in Theorem \ref{main_theorem_0} $(iii)$, provided that $u$ satisfies \textup{(INV)}.
  	
  	Now, we observe that $\A_{dif} \subset \A_{hom} \subset \A_0$. Indeed, the fact that Lipschitz maps (hence also $C^1$ maps) satisfy Lusin's condition (N) is well-known (see e.g.\ \cite{EvansGariepy1992_1st}), hence $\A_{dif} \subset \A_{hom}$. As for $\A_{hom} \subset \A_0$, this is a consequence of \cite[Theorem 5.5]{HenaoMora-Corral2012}. If $b$ satisfies Lusin's condition \textup{(N)}, then $b \in \A_{hom}$, and if we assume further that $b$ is a $C^1$ diffeomorphism, then $b \in \A_{dif}$. Therefore $\A_{hom}$ and $\A_{dif}$ are nonempty, provided that $b$ has enough regularity.
  	
  	To see that $\det Du \ge 0$ a.e.\ for every $u$ in $\lbar{\A_{dif}}$ and $\lbar{\A_{hom}}$, we note that each member $u\in\lbar{\A_{hom}}$ (and hence also each $u\in\lbar{\A_{dif}}$) is a weak $\W{1,n-1}$ limit of a sequence of homeomorphisms with a.e.\ positive Jacobians, hence \cite[Theorem 1]{HenclOnninen2018} implies that the Jacobian of $u$ is also positive a.e.
  \end{proof}
  
  \textit{Remark:} Suppose that $u_{dif}$ and $u_{hom}$ are minimizers of $E$ in the classes $\lbar{\A_{dif}}$ and $\lbar{\A_{hom}}$, respectively. Clearly, we have
  \[
  E(u_{hom}) \le E(u_{dif})
  \]
  since $\A_{dif} \subset \A_{hom}$. However, it is still not known if the inequality is strict or not. The same remark also applies to the inequality
  \[
  \inf_{\lbar{\A}} E \le \inf_{\A} E,
  \]
  where $\A\subset \A_0$.

  \paragraph{}
  It is natural to ask which subclasses $\A$ of $\A_0$ have the properties that allow us to conclude that condition \text{(INV)} always holds for each $u \in\lbar{\A}$ (or just for each $u \in\lbar{\A}$ that minimizes $E$ in its respective class) so that $u$ has the nice properties guaranteed by Lemma \ref{lemma_1}. This question is notoriously difficult due to the fact that in the $\W{1,n-1}(\om;\R^n)$ setting, condition \text{(INV)} does not pass to the limit in general, as shown by Conti and De Lellis in \cite{ContiDeLellis2003}. However, in a recent work \cite{DolezalovaHenclMaly2023}, Dole\v{z}alov\'a, Hencl, and Mal\'y have shown that by assuming that $H$ satisfies a stronger coercivity condition, condition (INV) does pass to the limit for any sequence $(u_j)_j$ of homeomorphisms with a fixed Dirichlet boundary condition such that $\sup_{j}E(u_j)<\oo$. 
  
 The additional conditions assumed in \cite{DolezalovaHenclMaly2023} are the following: there exist some $c>0$ such that
 \begin{equation}\label{condition_H1}
 	c^{-1}H(t) \le H(2t) \le c H(t)
 \end{equation}	
 and some $C>0$ such that 
 \begin{equation}\label{condition_H2}
 	 H(t) \ge \frac{C}{t^a} \quad\textup{for all}\quad t>0,
 \end{equation}
 where $a = \frac{n-1}{n^2-3n+1}$ (when $n=3$, we have $a=2$). Before we state our last result, it should be noted that the existence of a minimizer of $E$ in $\lbar{\A_{hom}}$ was already proved in \cite[Theorem 5.5]{DolezalovaHenclMolochanova2023} by Dole\v{z}alov\'a, Hencl, and Molchanova (with a less restrictive Dirichlet boundary condition than ours). The novelty of Theorem \ref{main_result_in_Hom} is the extra properties of $u$ that follow from Theorem \ref{main_theorem_0} $(iii)$. The existence of a minimizer $E$ in $\lbar{\A_{hom}}$ already follows from Theorem \ref{main_theorem_1} without having to assume that $H$ satisfies \eqref{condition_H1} and \eqref{condition_H2}, but such a minimizer may not satisfy condition (INV).

 \begin{theorem}\label{main_result_in_Hom}
 	Let $b \in \W{1,n-1}(\om;\R^n)$ be an orientation-preserving homeomorphism satisfying Lusin's condition \textup{(N)} and $E$ be an energy functional of the form \eqref{neohookean_energy}, where $H$ satisfies \eqref{condition_H0}, \eqref{condition_H1} and \eqref{condition_H2}. Then each function in $\lbar{\A_{hom}}$ satisfies all of the equivalent conditions in Lemma \ref{lemma_1}, and $E$ admits a minimizer in $\lbar{\A_{hom}}$.
 	
 	In particular, for the $n=3$ case, the energy functional 
 	\[
 	E(u) = \int_\om |Du(x)|^2 + H(\det Du(x)) \,dx,
 	\]
 	admits a minimizer in $\lbar{\A_{hom}}$, the weak closure of the class 
 	\begin{align*}
 		\A_{hom} =  \Big\{ u &\in \W{1,2}(\om;\R^3) : u \ \ \textup{is a homeomorphism,}\ \det Du > 0 \ \textup{a.e.,}\\
 		&u\ \text{satisfies Lusin's condition \textup{(N)},}\ u = b \ \textup{on }\om\bsl\td\om \text{ and }  E(u) \le E(b) \Big\},
 	\end{align*}
 	and for each $u \in \lbar{\A_{hom}}$, $u$ satisfies the divergence identities \eqref{div_identity}, $u$ satisfies condition \textup{(INV)}, $\Det Du = (\det Du)\L^3$, $\im_G(u,\om) = \om_b$ a.e. and $u^{-1} \in \W{1,1}(\om_b;\R^3)$.
 \end{theorem}
 
 \begin{proof}
 	Since $b \in \A_{hom}$, $\lbar{\A_{hom}}$ is nonempty and we can apply Theorem \ref{main_theorem_1}. From \cite[Theorem 1.1]{DolezalovaHenclMaly2023}, each element of $\lbar{\A_{hom}}$ satisfies condition (INV), hence Theorem \ref{main_theorem_0} $(iii)$ holds.
 \end{proof}
 It was shown in \cite[Theorem 5.5]{DolezalovaHenclMolochanova2023} that, under the assumptions \eqref{condition_H1} and \eqref{condition_H2}, the minimizer of $E$ in $\lbar{\A_{hom}}$ satisfies Lusin's condition (N) as well.

 \paragraph{}
 In fact, as mentioned in \cite{DolezalovaHenclMaly2023}, the coercivity condition \eqref{condition_H2} can be relaxed to a uniform $\Leb{\frac{1}{n-1}}$ bound on the \textit{distortion function}, which is defined as follows:
 Let $K_u(x) := \frac{|Du(x)|^n}{|\det Du(x)|}$ be the \textit{distortion function} of a Sobolev function $u$, where we set $K_u(x) = 1$ if $\det Du(x) = 0$ (here $|Du(x)|$ denotes the operator norm of the matrix $Du(x)$). \cite[Theorem 3.1]{DolezalovaHenclMaly2023} states that for a sequence of Sobolev homeomorphisms $u_j$ with uniformly bounded energy ($\sup_j  E(u_j) < \oo$), if the sequence $(K_{u_j})_j$ is bounded in the $\Leb{\frac{1}{n-1}}$ norm, then condition (INV) passes to the limit. Indeed, the fact that $\sup_j \nm{K_{u_j}}_{\Leb{\frac{1}{n-1}}} < \oo$ follows from $\sup_j  E(u_j) < \oo$ and  \eqref{condition_H2} is a consequence of the  H\"{o}lder-Young inequality 
 \begin{align*}
 	\int_\om K_u^{\frac{1}{n-1}}(x) \,dx &= \int_\om |Du(x)|^{\frac{n}{n-1}}\ |\det Du(x)|^{-\frac{1}{n-1}} \,dx \\
 	&\le \frac{1}{p} \int_\om |Du(x)|^{n-1}  \,dx + \frac{1}{q} \int_\om  |\det Du(x)|^{-\frac{n-1}{n^2-3n+1}} \,dx,
 \end{align*}
 where $p = \frac{(n-1)^2}{n}$ and $q = \frac{(n-1)^2}{n^2-3n+1}$. In particular, for $n=3$, we have $\frac{n-1}{n^2-3n+1}=2$.

 \paragraph{Acknowledgement} I thank Marco Barchiesi, Duvan Henao, Carlos Mora-Corral, and R\'emy Rodiac for our discussions and suggestions. Also, I am indebted to Jan Kristensen for introducing me to this beautiful area of mathematics, the calculus of variations.

	\bibliography{Sobolev_regularity_of_the_inverse_for_minimizers_of_the_neo-Hookean_ver2}

\begin{thebibliography}{10}

\bibitem{Ball1977}
{\sc J.~M. Ball}, {\em Convexity conditions and existence theorems in nonlinear
  elasticity}, Arch. Ration. Mech. Anal., 63 (1977), pp.~337--403.

\bibitem{Ball1981}
\leavevmode\vrule height 2pt depth -1.6pt width 23pt, {\em Global invertibility
  of {Sobolev} functions and the interpenetration of matter}, Proc. R. Soc.
  Edinb., Sect. A, Math., 88 (1981), pp.~315--328.

\bibitem{Ball1982}
\leavevmode\vrule height 2pt depth -1.6pt width 23pt, {\em Discontinuous
  equilibrium solutions and cavitation in nonlinear elasticity}, Philos. Trans.
  R. Soc. Lond., Ser. A, 306 (1982), pp.~557--611.

\bibitem{BallCurrieOlver1981}
{\sc J.~M. Ball, J.~C. Currie, and P.~J. Olver}, {\em Null {Lagrangians}, weak
  continuity, and variational problems of arbitrary order}, J. Funct. Anal., 41
  (1981), pp.~135--174.

\bibitem{Ball1984}
{\sc J.~M. Ball and F.~Murat}, {\em {{\(W^{1,p}\)}}-quasiconvexity and
  variational problems for multiple integrals}, J. Funct. Anal., 58 (1984),
  pp.~225--253.

\bibitem{BarchiesiHenaoMora-Corral2017}
{\sc M.~Barchiesi, D.~Henao, and C.~Mora-Corral}, {\em Local invertibility in
  {S}obolev spaces with applications to nematic elastomers and
  magnetoelasticity}, Arch. Ration. Mech. Anal., 224 (2017), pp.~743--816.

\bibitem{BarchiesiHenaoMora-CorralRodiac2023_harmonic}
{\sc M.~Barchiesi, D.~Henao, C.~Mora-Corral, and R.~Rodiac}, {\em Harmonic
  dipoles and the relaxation of the neo-{Hookean} energy in {3D} elasticity},
  Arch. Ration. Mech. Anal., 247 (2023), p.~46.

\bibitem{BarchiesiHenaoMora-CorralRodiac2023_relaxation}
{\sc M.~Barchiesi, D.~Henao, C.~Mora-Corral, and R.~Rodiac}, {\em A relaxation
  approach to the minimisation of the neo-{H}ookean energy in 3{D}}.
\newblock SIAM J. Math. Anal., 2023.
\newblock ArXiv:2311.02952.

\bibitem{Canulef-AguilarHenao2019}
{\sc V.~{C}a\~{n}ulef {Aguilar} and D.~Henao}, {\em A lower bound for the void
  coalescence load in nonlinearly elastic solids}, Interfaces Free Bound., 21
  (2019), pp.~409--440.

\bibitem{ContiDeLellis2003}
{\sc S.~Conti and C.~De~Lellis}, {\em Some remarks on the theory of elasticity
  for compressible {N}eohookean materials}, Ann. Sc. Norm. Super. Pisa Cl. Sci.
  (5), 2 (2003), pp.~521--549.

\bibitem{DeLellisGhiraldin2010}
{\sc C.~De~Lellis and F.~Ghiraldin}, {\em An extension of the identity {Det} =
  det}, C. R., Math., Acad. Sci. Paris, 348 (2010), pp.~973--976.

\bibitem{DolezalovaHenclMaly2023}
{\sc A.~Dole\v{z}alov\'{a}, S.~Hencl, and J.~Mal\'{y}}, {\em Weak {L}imit of
  {H}omeomorphisms in {$W^{1,n-1}$} and ({INV}) {C}ondition}, Arch. Ration.
  Mech. Anal., 247 (2023), p.~80.

\bibitem{DolezalovaHenclMolochanova2023}
{\sc A.~Doležalová, S.~Hencl, and A.~Molchanova}, {\em Weak limit of
  homeomorphisms in ${W}^{1,n-1}$: invertibility and lower semicontinuity of
  energy}.
\newblock ArXiv:2212.06452, 2023.

\bibitem{EvansGariepy1992_1st}
{\sc L.~C. Evans and R.~F. Gariepy}, {\em Measure theory and fine properties of
  functions}, Boca Raton: CRC Press, 1992.

\bibitem{Federer1996_GMT}
{\sc H.~Federer}, {\em Geometric measure theory.}, Class. Math., Berlin:
  Springer-Verlag, repr. of the 1969 ed.~ed., 1996.

\bibitem{Gent1990}
{\sc A.~N. Gent}, {\em {Cavitation in Rubber: A Cautionary Tale}}, Rubber Chem.
  Technol., 63 (1990), pp.~49--53.

\bibitem{GiaquintaGiuseppeSoucek1989}
{\sc M.~Giaquinta, G.~Modica, and J.~Sou{\v{c}}ek}, {\em Cartesian currents,
  weak diffeomorphisms and existence theorems in nonlinear elasticity}, Arch.
  Ration. Mech. Anal., 106 (1989), pp.~97--159.

\bibitem{Henao2009}
{\sc D.~Henao}, {\em Cavitation, invertibility, and convergence of regularized
  minimizers in nonlinear elasticity}, J. Elasticity, 94 (2009), pp.~55--68.

\bibitem{HenaoMora-Corral2010}
{\sc D.~Henao and C.~Mora-Corral}, {\em Invertibility and weak continuity of
  the determinant for the modelling of cavitation and fracture in nonlinear
  elasticity}, Arch. Ration. Mech. Anal., 197 (2010), pp.~619--655.

\bibitem{HenaoMora-Corral2011}
\leavevmode\vrule height 2pt depth -1.6pt width 23pt, {\em Fracture surfaces
  and the regularity of inverses for {BV} deformations}, Arch. Ration. Mech.
  Anal., 201 (2011), pp.~575--629.

\bibitem{HenaoMora-Corral2012}
\leavevmode\vrule height 2pt depth -1.6pt width 23pt, {\em Lusin's condition
  and the distributional determinant for deformations with finite energy}, Adv.
  Calc. Var., 5 (2012), pp.~355--409.

\bibitem{HenaoMora-Corral2015}
\leavevmode\vrule height 2pt depth -1.6pt width 23pt, {\em Regularity of
  inverses of {Sobolev} deformations with finite surface energy}, J. Funct.
  Anal., 268 (2015), pp.~2356--2378.

\bibitem{HenaoMora-CorralOliva2021}
{\sc D.~Henao, C.~Mora-Corral, and M.~Oliva}, {\em Global invertibility of
  {Sobolev} maps}, Adv. Calc. Var., 14 (2021), pp.~207--230.

\bibitem{HenaoMora-CorralXu2016}
{\sc D.~Henao, C.~Mora-Corral, and X.~Xu}, {\em A numerical study of void
  coalescence and fracture in nonlinear elasticity}, Comput. Methods Appl.
  Mech. Eng., 303 (2016), pp.~163--184.

\bibitem{HenaoSerfaty2013}
{\sc D.~Henao and S.~Serfaty}, {\em Energy estimates and cavity interaction for
  a critical-exponent cavitation model}, Commun. Pure Appl. Math., 66 (2013),
  pp.~1028--1101.

\bibitem{HenclMaly2010}
{\sc S.~Hencl and J.~Mal{\'y}}, {\em Jacobians of {Sobolev} homeomorphisms},
  Calc. Var. Partial Differ. Equ., 38 (2010), pp.~233--242.

\bibitem{HenclOnninen2018}
{\sc S.~Hencl and J.~Onninen}, {\em Jacobian of weak limits of {S}obolev
  homeomorphisms}, Adv. Calc. Var., 11 (2018), pp.~65--73.

\bibitem{HorganPolignone1995}
{\sc C.~O. Horgan and D.~A. Polignone}, {\em Cavitation in {Nonlinearly}
  {Elastic} {Solids}: {A} {Review}}, Appl. Mech. Rev., 48 (1995), pp.~471--485.

\bibitem{Kristensen2015}
{\sc J.~Kristensen}, {\em A necessary and sufficient condition for lower
  semicontinuity}, Nonlinear Anal., Theory Methods Appl., Ser. A, Theory
  Methods, 120 (2015), pp.~43--56.

\bibitem{Kroemer2020}
{\sc S.~Kr{\"o}mer}, {\em Global invertibility for orientation-preserving
  {Sobolev} maps via invertibility on or near the boundary}, Arch. Ration.
  Mech. Anal., 238 (2020), pp.~1113--1155.

\bibitem{KumarLopez-Pamies2021}
{\sc A.~Kumar and O.~Lopez-Pamies}, {\em The poker-chip experiments of {Gent
  and Lindley} (1959) explained}, J. Mech. Phys. Solids, 150 (2021), p.~104359.

\bibitem{Llavona1986_approx_book}
{\sc J.~G. Llavona}, {\em Approximation of continuously differentiable
  functions}, vol.~130 of North-Holland Math. Stud., Elsevier, Amsterdam, 1986.

\bibitem{Maly1993}
{\sc J.~Mal{\'y}}, {\em Weak lower semicontinuity of polyconvex integrals},
  Proc. R. Soc. Edinb., Sect. A, Math., 123 (1993), pp.~681--691.

\bibitem{Mora-Corral2014}
{\sc C.~Mora-Corral}, {\em Quasistatic evolution of cavities in nonlinear
  elasticity}, SIAM J. Math. Anal., 46 (2014), pp.~532--571.

\bibitem{Mueller1988}
{\sc S.~M{\"u}ller}, {\em Weak continuity of determinants and nonlinear
  elasticity. ({Continuit{\'e}} faible des d{\'e}terminants et applications
  {\`a} l'{\'e}lasticit{\'e} non lin{\'e}aire)}, C. R. Acad. Sci., Paris,
  S{\'e}r. I, 307 (1988), pp.~501--506.

\bibitem{MullerTangYan1994}
{\sc S.~M{\"u}ller, T.~Qi, and B.~S. Yan}, {\em On a new class of elastic
  deformations not allowing for cavitation}, Ann. Inst. Henri Poincar{\'e},
  Anal. Non Lin{\'e}aire, 11 (1994), pp.~217--243.

\bibitem{MullerSpector1995}
{\sc S.~M\"{u}ller and S.~J. Spector}, {\em An existence theory for nonlinear
  elasticity that allows for cavitation}, Arch. Rational Mech. Anal., 131
  (1995), pp.~1--66.

\bibitem{Negron-MarreroSivaloganathan2024}
{\sc P.~V. Negrón–Marrero and J.~Sivaloganathan}, {\em Cavitation of a
  spherical body under mechanical and self-gravitational forces}, Proc. R. Soc.
  Edinb. A: Math.,  (2024), pp.~1--31.

\bibitem{SivaloganathanSpector2000}
{\sc J.~Sivaloganathan and S.~J. Spector}, {\em On the existence of minimizers
  with prescribed singular points in nonlinear elasticity}, J. Elasticity, 59
  (2000), pp.~83--113.

\bibitem{SivaloganathanSpector2000a}
\leavevmode\vrule height 2pt depth -1.6pt width 23pt, {\em On the optimal
  location of singularities arising in variational problems of nonlinear
  elasticity}, J. Elasticity, 58 (2000), pp.~191--224.

\bibitem{Sverak1988}
{\sc V.~{\v{S}}ver{\'a}k}, {\em Regularity properties of deformations with
  finite energy}, Arch. Ration. Mech. Anal., 100 (1988), pp.~105--127.

\end{thebibliography}

\end{document}